\documentclass{amsart}

\newtheorem{theorem}{Theorem}
\newtheorem{lemma}[theorem]{Lemma}
\newtheorem{corollary}[theorem]{Corollary}
\newtheorem{proposition}[theorem]{Proposition}

\theoremstyle{definition}
\newtheorem{definition}[theorem]{Definition}

\newtheorem{remark}[theorem]{Remark}

\usepackage{graphics, graphicx, xcolor}
\usepackage{graphicx}
\graphicspath{ {images/} }
\definecolor{darkred}{rgb}{0.7,0,0} 
\newcommand{\darkred}{\color{darkred}} 
\newcommand{\defn}[1]{\emph{\darkred #1}} 
\newcommand{\begeq}{\begin{equation}}
\newcommand{\eneq}{\end{equation}}

\usepackage{relsize}
\usepackage[colorinlistoftodos]{todonotes}
\usepackage{lipsum,  verbatim}
\usepackage{mathtools}
\usepackage{amsfonts}
\usepackage{txfonts}

\begin{document}

\title{Cores with distinct parts and bigraded Fibonacci numbers}

\author{Kirill Paramonov}\email{kbparamonov@ucdavis.edu}

\begin{abstract}
The notion of $(a,b)$-cores is closely related to rational $(a,b)$ Dyck paths due to Anderson's bijection, and thus the number of $(a,a+1)$-cores is given by the Catalan number $C_a$. Recent research shows that $(a,a+1)$ cores with distinct parts are enumerated by another important sequence- Fibonacci numbers $F_a$. In this paper, we consider the abacus description of $(a,b)$-cores to introduce the natural grading and generalize this result to $(a,as+1)$-cores. We also use the bijection with Dyck paths to count the number of $(2k-1,2k+1)$-cores with distinct parts. We give a second grading to Fibonacci numbers, induced by bigraded Catalan sequence $C_{a,b} (q,t)$.
\end{abstract}

\maketitle

\section{Introduction}

For two coprime integers $a$ and $b$, the rational Catalan number $C_{a,b}$ and its bigraded generalization $C_{a,b}(q,t)$ have caught the attention of different researchers due to their connection to algebraic combinatorics and geometry \cite{ALW.16, AHJ.14, GM.16, GMV.16}. Catalan numbers can be analyzed from the perspective of different combinatorial objects: rational $(a,b)$-Dyck paths, simultaneous $(a,b)$-core partitions and abacus diagrams. 

In 2015, Amdeberhan~\cite{Amdeberhan.15} conjectured that the number of $(a,a+1)$-cores with distinct parts is equal to the Fibonacci number $F_{a+1}$, and also conjectured the formulas for the largest size and the average size of such partitions. This conjecture has been proven by H.Xiong~\cite{Xiong.15}:

\begin{theorem} (Xiong,15)
\label{thm.s=1}
For $(a,a+1)$-core partitions with distinct parts, we have
\begin{enumerate}
\item the number of such partitions equals to the Fibonacci number $F_{a+1}$;
\item the largest size of such partition is $\big\lfloor \frac{1}{3} \binom{a+1}{2} \big\rfloor$;
\item there are $\frac{3-(-1)^{a \mod 3}}{2}$ such partitions of maximal size;
\item the total number of these partitions and the average sizes are, respectively, given by
$$\sum_{i+j+k=a+1} F_i F_j F_k \quad \text{and} \quad \sum_{i+j+k=a+1} \frac{F_i F_j F_k}{F_{a+1}}.$$
\end{enumerate}
\end{theorem} 

Part (1) of the above theorem was independently proved by A.Straub~\cite{Straub.16}.

Another interesting conjecture of Amdeberhan is the number of $(2k-1, 2k+1)$- cores with distinct parts. This conjecture have been proven by Yan, Qin, Jin and Zhou~\cite{YQJZ.16}:

\begin{theorem} (YQJZ,16)
\label{thm.r=2}
The number of $(2k-1,2k+1)$-cores with distinct parts is equal to $2^{2k-2}$.
\end{theorem}

The proof uses somewhat complicated arguments about the poset structure of cores. Results by Zaleski and Zeilberger~\cite{ZZ.16} improve the argument using Experimental Mathematics tools in Maple. More recently Baek, Nam and Yu provided simpler bijective proof in~\cite{BNY.17}.

Another set of combinatorial objects that caught the attention of a number of researchers \cite{Straub.16,NS.16,Zaleski.17} is the set of $(a,as-1)$-cores with distinct parts. In particular, there is a Fibonacci-like recursive relation for the number of such cores:

\begin{theorem} (Straub, 16) 
The number $N_s (a)$ of $(a, as-1)$-core partitions into distinct parts is characterized by $N_s(1) = 1, \ N_s(2) = s$ and, for $a \ge 3$,
$$N_s(a) = N_s(a-1) + s N_s(a-2).$$
\end{theorem} 

In this paper, we analyze simultaneous core partitions in the context of Anderson's bijection and in Section \ref{section.description} we provide a simple description of the set of $(a,as+1)$-cores with distinct parts in terms of abacus diagrams, which also allows us to provide another proof of Theorem~\ref{thm.s=1} parts (1), (2) and (3) in Section \ref{sec.s=1}. 

In Section~\ref{sec.r=2} we use the connection between cores and Dyck paths to provide another simple proof of Theorem~\ref{thm.r=2}. 

In Section~\ref{sec.r=1} we introduce graded Fibonacci numbers
\begin{equation*}
F_{a,b}(q) = \sum_{\kappa} q^{area(\kappa)},
\end{equation*}
where the sum is taken over all $(a,b)$-cores $\kappa$ with distinct parts and $area$ is some statistic on $(a,b)$-cores. We show that $F_{a,a+1} (1) = F_{a+1}$- the regular Fibonacci sequence, and prove recursive relations for $F^{(s)}_a (q) := F_{a,as+1} (q)$. Using properties of $F_{a,a+1} (q)$ we provide another proof of Theorem~\ref{thm.s=1} part (4). 

In Section~\ref{section.bigraded} we introduce bigraded Fibonacci numbers as a summand of bigraded Catalan numbers:
\begin{equation*}
F^{(s)}_a (q,t) = \sum_{\pi} q^{area(\pi)} t^{bounce(\pi)},
\end{equation*}
where the sum is taken over all $(a,as+1)$- Dyck paths corresponding to $(a,as+1)$-cores with distinct parts, and statistics $(area, bounce)$ are two standard statistics on Dyck paths (see~\cite{Loehr.03}).

Using abacus diagrams, we can get a simple formula for $F^{(s)}_a (q,t)$ and prove a theorem that gives recursive relations similar to the recursive relations for regular Fibonacci numbers. We use the standard notation $(s)_r = 1+r +\ldots+r^{s-1}$.
\begin{theorem}
Normalized bigraded Fibonacci numbers $\tilde F^{(s)}_a (q,t)$ satisfy the recursive relations
\begin{equation*}
\tilde F^{(s)}_{a+1} (q,t) = \tilde F^{(s)}_{a} (q,t) + qt^a \left(s\right)_{qt^a} \tilde F^{(s)}_{a-1} (q,t) = \tilde F^{(s)}_{a} (qt,t) + qt \left(s\right)_{qt} \tilde F^{(s)}_{a-1} (qt^2,t),
\end{equation*}
with initial conditions $\tilde F^{(s)}_0 (q,t) =\tilde F^{(s)}_1(q,t) = 1$.
\end{theorem}

\subsection*{Acknowledgments}
The author would like to thank Evgeny Gorskiy, Anne Schilling and Tewodros Amdeberhan for suggesting the problem and providing helpful discussions and comments.
This research was partially supported by NSF grant DMS-1500050.

\section{Background and notation}
\label{section.background}

For two coprime numbers $a$ and $b$ consider a rectangle $R_{a,b}$ on square lattice with bottom-left corner at the origin and top-right corner at $(a,b)$. We call the diagonal from $(0,0)$ to $(a,b)$ the \defn{main diagonal} of the rectangle $R_{a,b}$. An \defn{$(a,b)$-Dyck path} is a lattice path from $(0,0)$ to $(a,b)$ that consists of North and East steps and that lies weakly above the main diagonal. Denote the set of $(a,b)$-Dyck paths by $\mathrm{D}_{a,b}$.
 
For a box in $R_{a,b}$ with bottom-right corner coordinates $(x,y)$, define the \defn{rank of the box} to be equal to $ay-bx$ (see Fig.~\ref{figure.anderson}, left). Note that a box has positive rank if and only if it lies above the main diagonal. For a rational Dyck path $\pi$, we define the \defn{area statistic} $area(\pi)$ to be the number of boxes in $R_{a,b}$ with positive ranks that are below $\pi$. 

\begin{figure}[h]
\includegraphics[scale=0.38]{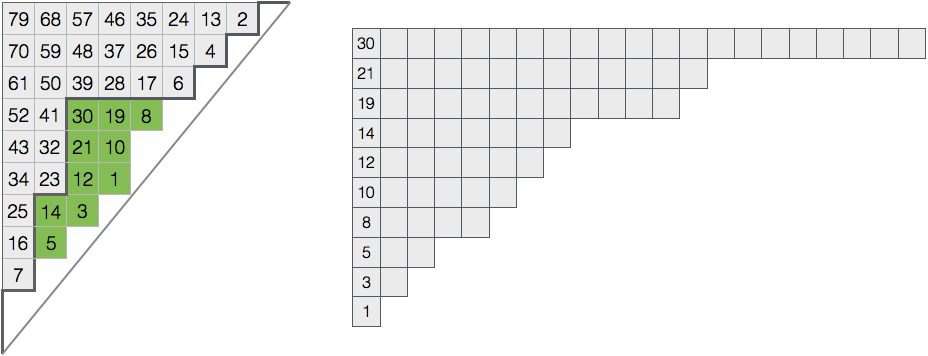} \centering
\caption{$(9,11)$-Dyck path $\pi$ and $(9,11)$-core $\kappa$ with $\mathbf{core}(\pi) = \kappa$.} \centering
\label{figure.anderson}
\end{figure}

Denote the set of ranks of all the area boxes of $\pi$ as \defn{$\alpha(\pi)$}. Note that $\alpha(\pi)$ doesn't contain any multiples of $a$ or $b$ and it has an \defn{$(a,b)$-nested property}, i.e.
\begin{equation}
\label{equation.nested}
(i \in \alpha(\pi),\ i>a) \Rightarrow i-a \in \alpha(\pi), \quad (j \in \alpha(\pi),\ j>b) \Rightarrow j-b \in \alpha(\pi).
\end{equation}

Also note that $\alpha(\pi)$ completely determines the Dyck path $\pi$.

\begin{remark}
The $(a,b)$-nested property of $\alpha(\pi)$ is equivalent to the $(a,b)$-invariant property of the complement of $\alpha(\pi)$ (see~\cite{GMV.16}).
\end{remark}

Consider $i \in \{0,\ldots, a-1\}$. If we can find a column in $R_{a,b}$ with a box of rank $i$, define $e_i (\pi)$ to be the number of boxes in that column below $\pi$ and above the main diagonal. If there is no box of rank $i$ (i.e. if $b|i$), define $e_i(\pi)$ to be zero.  Note that 
$$e_i(\pi) = \big| \big\{ x \in \alpha(\pi)\ |\ x  \equiv i \ (\mathrm{mod}\ a) \big\} \big|.$$ 
The vector $e(\pi) = (e_0 (\pi), \ldots, e_{a-1} (\pi))$ is defined to be the \defn{area vector} of $\pi$. Note that $e_0(\pi)$ is always zero.

A \defn{partition} $\lambda$ of $n$ is a finite non-increasing sequence $(\lambda_1, \lambda_2, \ldots, \lambda_l)$ of positive integers, which sum up to $n$. Integers $\lambda_i$ are called \defn{parts} of the partition $\lambda$ and $n$ is called the \defn{size} of the partition. A partition $\lambda$ is sometimes represented by its Young diagram (we will use English notation) (see Fig.~\ref{figure.anderson}, right). The \defn{hook length} of a box in the Young diagram of $\lambda$ is defined to be the number of boxes directly below and directly to the right of the given box, including that box itself. 

We say that a partition $\kappa$ is an \defn{$(a,b)$-core} if there are no boxes in the Young diagram of $\kappa$ with hook length equal to $a$ or $b$. Denote the set of all $(a,b)$-cores as $\mathrm{K}_{a,b}$. 
For example, a partition $\kappa = (21, 13, 12, 8, 7, 6, 5, 3, 2,1)$ belongs to $\mathrm{K}_{9,11}$ (see Fig.~\ref{figure.anderson}).

Define the size statistic on cores \defn{$\mathrm{size}(\kappa)$} to be the sum of all parts of $\kappa$.
Given an $(a,b)$-core $\kappa$, denote the set of hook lengths of the boxes in the first column by \defn{$\beta(\kappa)$}. It is not hard to show that $\beta(\kappa)$ also satisfies $(a,b)$-nested property (\ref{equation.nested}) (in fact, the nested condition of $\beta(\kappa)$ and the fact that $\beta(\kappa)$ doesn't contain multiples of $a$ or $b$ are sufficient for $\kappa$ to be an $(a,b)$-core). Note that the set $\beta(\kappa)$ completely determines the core $\kappa$.

For two coprime numbers $a$ and $b$, there is a bijection \defn{$\mathbf{path}$}$\colon\mathrm{K}_{a,b} \to \mathrm{D}_{a,b}$ due to J.Anderson~\cite{Anderson.02}.  We can describe the map $\mathbf{path}$ by specifying how it acts on sets $\alpha$ and $\beta$, namely $\alpha(\mathbf{path} (\kappa)) = \beta(\kappa)$. 

We also denote the map \defn{$\mathbf{core}$}$\colon\mathrm{D}_{a,b}\to\mathrm{K}_{a,b}$ by $\mathbf{core} = \mathbf{path}^{-1}$. It follows that $\beta(\mathbf{core}(\pi)) = \alpha(\pi)$. 

Define the corresponding area statistic on cores as \defn{$area(\kappa)$}$~:=area(\mathbf{path} (\kappa))$. Note that $area(\kappa) = \left\vert{\alpha(\mathbf{path}(\kappa))}\right\vert=\left\vert{\beta(\kappa)}\right\vert$ is equal to the number of rows in partition $\kappa$. Similarly, an \defn{area vector} of $\kappa$ is defined as $c(\kappa) = \left(c_0 (\kappa),\ldots, c_{a-1} (\kappa)\right)$ with $c_i(\kappa)$ equal to the number of elements in $\beta(\kappa)$ with residue $i$ modulo $a$. 

The third construction is the abacus diagram of an $a$-core. An \defn{$a$-abacus diagram} consists of $a$ rows called runners indexed by $\{0,1, \ldots, a-1\}$. Each row represents a real line with integer positions filled with black or white beads (see Fig.~\ref{figure.abacus}, right).

There is a correspondence between partitions and abacus diagrams. Let $\lambda$ be the Young diagram of a partition in Russian notation with each row going in North-East direction (see Fig.~\ref{figure.abacus}, left). Reading from left to right, the boundary of $\lambda$ (denoted by \defn{$\partial(\lambda)$}) consists of NE-steps and SE-steps. We enumerate the steps so that the first NE-step has number 0, and the enumeration is consecutively increasing for the steps going from left to right. 

A corresponding $a$-abacus diagram is constructed by filling position $k$ on runner $i$ with a black bead if the step $i + ak$ of $\partial(\lambda)$ is an SE-step and filling that position with a white bead otherwise. According to our construction all negative positions on runners are filled with black beads and $0$-th position on $0$-th runner is always filled with a white bead. The importance of that construction becomes clear when we let partition $\lambda$ to be an $a$-core.

\begin{figure}[t]
\includegraphics[scale=0.38]{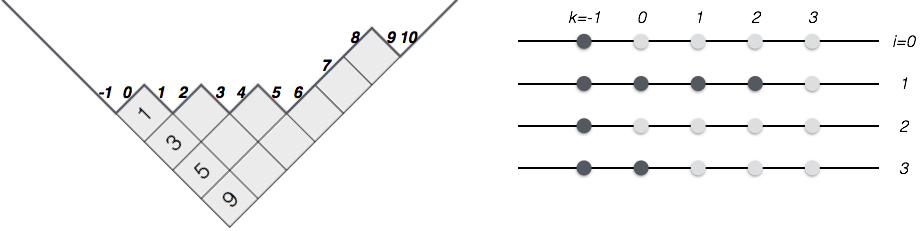} \centering
\caption{$(4,13)$-core $(6,3,2,1)$ and the corresponding abacus diagram with $d=(0,3,0,1)$.} \centering
\label{figure.abacus}
\end{figure}

\begin{proposition}
\label{prop.core}
Partition $\lambda$ is an $a$-core if and only if for any SE-step $j$ in $\partial(\lambda)$ step $j-a$ is also an SE-step.
\end{proposition}
\begin{proof}
Given a box $x$ of a Young diagram in Russian notation of partition $\lambda$, there is an SE-step of $\partial(\lambda)$ directly in the NE direction of $x$ with number $se(x)$ and there is an NE-step of $\partial(\lambda)$ directly in the SE direction of $x$ with number $ne(x)$. The hook length of box $x$ is then given by $se(x)-ne(x)$.

($\Leftarrow$)  If for every SE-step $j$ the step $j-a$ is also an SE-step, there is no box $x$ with $se(x)-ne(x) = a$ and $\lambda$ is an $a$-core.

($\Rightarrow$) Conversely, for any pair of SE-step $j$ and NE-step $k$ with $j > k$, there exists a box $x$ with $se(x) = j$ and $ne(x) = k$. Thus, if there are no boxes $x$ of hook length $a$, there are no pairs $(j,k)$ with $j-k=a$, and for any SE-step $j$ the step $j-a$ must be also SE.
\end{proof}

\begin{corollary}
\label{corollary.acore}
A partition $\lambda$ is an $a$-core if and only if the set of black beads in the corresponding $a$-abacus diagram is left-justified, i.e. for any runner $i$ there exist an integer $d_i \ge 0$ such all positions $k < d_i$ are filled with black beads and all positions $k \ge d_i$ are filled with white beads.
\end{corollary}

\begin{proof}
Remember that a black bead of an abacus diagram of $\lambda$ on a runner $i$ in position $k$ corresponds to an SE-step $j = i + ak$ of $\partial(\lambda)$. 

Note that an abacus diagram is left-justified if and only if for any black bead on runner $i$ in position $k$, the bead in position $k-1$ is also a black one. In turn, that is equivalent to the fact that for any SE-step $j = i+ak$ of $\partial(\lambda)$ the step $j-a = i+ a(k-1)$ is also an SE-step.
\end{proof}

Given integers $d_i$ from Corollary~\ref{corollary.acore}, denote $d = (d_0, d_1, \ldots, d_{a-1})$. It is useful to think about $d_i$ as a number of black beads in nonnegative positions of a runner $i$. We will call $d$ the \defn{abacus vector} of an $a$-core $\lambda$.

Denote the map from $a$-cores $\lambda$ to integer $a$-dimensional abacus vectors $d$ by \defn{$\mathbf{abac}(\lambda)$}.

\begin{proposition}
Given an $(a,b)$-core $\kappa$, the $a$-abacus vector $(d_1,\ldots,d_{a-1}) = \mathbf{abac}(\kappa)$ is equal to the area vector $c(\kappa)$.
\end{proposition}

\begin{proof}
We'll use the notation from Proposition~\ref{prop.core}. 
Let $\{x_1, \ldots, x_l\}$ be the set of boxes in the first column of the Young diagram of $\kappa$. 
Then $ne(x_i) = 0$ for any $i$ and the set $\{se(x_1), \ldots, se(x_l)\}$ covers all positive SE-steps of $\partial(\kappa)$.

The hook length of $x_i$ is thus equal to $se(x_i) - ne(x_i) = se(x_i)$, and the set of hook lengths of $\{x_1,\ldots, x_l\}$ is equal to $\{se(x_1), \ldots, se(x_l)\}$.

By definition, the set of hook lengths of $\{x_1,\ldots, x_l\}$ is $\beta(\kappa)$.
The number of elements in $\beta(\kappa)$ with residue $i$ modulo $a$ is equal to the number of positive SE-steps in $\partial(\kappa)$ with residue $i$ modulo $a$, so $c_i(\kappa) = d_i$ and $c(\kappa) = d$.
\end{proof}

\begin{corollary}
\label{path.area.abac}
Given $\kappa\in \mathrm{K}_{a,b}$, the abacus vector $(d_1,\ldots,d_{a-1}) = \mathbf{abac}(\kappa)$ is equal to the area vector $e(\pi)$ of the Dyck path $\pi = \mathbf{path} (\kappa)$.
\end{corollary}

From our definition $d_0$ is always equal to 0, and thus we often omit that coordinate. For an arbitrary $a$-core $\lambda$ the only condition on the other coordinates $(d_1,\ldots, d_{a-1})$ is that they are all non-negative. It is often useful to think about $(d_1,\ldots, d_{a-1})$ as coordinates of an $a$-core in $\mathbb{Z}^{a-1}$.

\begin{remark}
Embedding of the set of $(a,b)$-cores in $\mathbb{Z}^{a-1}$ has been studied from the point of view of Ehrhart theory in~\cite{Johnson.15}.
\end{remark}

To describe the area statistic in terms of vectors $d$, notice that the number of rows in $\lambda$ is equal to the number of positive SE-steps in $\partial(\lambda)$, which correspond to non-negative black beads in the $a$-abacus diagram. Thus, we define the area statistic of $d$ to be \defn{$area(d)$} $ = \sum d_i$.

\section{Simultaneous cores with distinct parts.}
\label{section.description}
Let us consider an $a$-core $\lambda$. We can express the condition that $\lambda$ has distinct parts in terms of the boundary $\partial(\lambda)$.

\begin{proposition}
An $a$-core $\lambda$ has distinct parts if and only if for each positive SE-step $j$ in $\partial(\lambda)$, the steps $j-1$ and $j+1$ are NE-steps.
\end{proposition}

\begin{proof}
Suppose there are two consecutive SE-steps $j+1$ and $j$ in $\partial(\lambda)$. Then there are two rows $row_i$ and $row_{i+1}$ in the Young diagram of $\lambda$ that are bordered by steps $j+1$ and $j$ from NE-side. Then $\lambda_i = \mathrm{length}(row_i)$ and $\lambda_{i+1} = \mathrm{length}(row_{i+1})$ are equal, and thus we arrive to a contradiction.
\end{proof}

That property can be formulated in terms of vectors $d=\mathbf{abac}(\lambda)$. Given a subset $S$ of $\{0,1,\ldots, a-1\}$, we call $S$ to be an \defn{$a$-sparse set} if for any two elements $n,m \in S,\ |n-m| \neq 1$. The \defn{support} of the vector $d$ (denoted by \defn{$\mathrm{supp}(d)$}) is defined to be the set of all indexes $i$ with $d_i > 0$. We call vectors $d = (d_{0}, d_{1}, \ldots, d_{a-1})$ with sparse support to be \defn{$a$-sparse vectors}. Note that omitting $d_{0} = 0$ in $d$ doesn't change the sparsity of the vector. 

\begin{proposition}
\label{prop.distinct}
An $a$-core $\lambda$ has distinct parts if and only if the vector $d=\mathbf{abac}(\lambda)$ is an $a$-sparse vector.
\end{proposition}

\begin{proof}
($\Rightarrow$) Consider an index $i \in \{1,\ldots, a-1\}$ such that $d_i > 0$. For the corresponding $a$-abacus diagam that means the bead on runner $i$ in position $0$ is a black one, and the $i$-th step of $\partial(\lambda)$ is an SE-step. If $\lambda$ has distinct parts, that means steps $i-1$ and $i+1$ of $\partial(\lambda)$ are NE-steps and thus runners $i-1$ and $i+1$ don't have black beads on nonnegative positions, and $d_{i-1} = d_{i+1} = 0$. 

($\Leftarrow$) Conversely, suppose $d=\mathbf{abac}(\lambda)$ is a sparse vector. Any step SE-step $j$ of $\partial(\lambda)$ corresponds to some black bead on runner $i$ in position $k$. Because of the sparsity, beads on runners $i-1$ and $i+1$ in position $k$ must be white ones, and thus steps $j-1$ and $j+1$ of $\partial(\lambda)$ must be NE-steps.
Note that for any SE-step $j$ corresponding to the runner $i=1$, step to the left of $j$ is on the $0$-th runner and thus is always an NE-step. Similarly, when $i=a-1$ the step to the right of $j$ is on the $0$-th runner in positive position and thus is always an NE-step.
\end{proof}

We now consider an additional structure of simultaneous $(a,b)$-cores $\kappa$ in terms of their $a$-abacus diagrams. We use the Proposition~\ref{prop.core} one more time, but now looking at $\kappa$ as a $b$-core.

\begin{proposition}
\label{prop.sim}
Let $\kappa$ be an $a$-core with $d = \mathbf{abac}(\kappa)$ and consider positive integer number $b=sa+r$ with $0\le r<a$. Then $\kappa$ is a simultaneous $(a,b)$- core if and only if for any index $i$ between 1 and $a-1$, one of the following is true:
\begin{enumerate}
\item $i \geq r$ and $d_{i} \leq d_{i-r} + s$,
\item $i<r$ and $d_{i} \leq d_{i+a-r}+s+1$.
\end{enumerate}
\end{proposition}

\begin{proof} ($\Rightarrow$)
Fix an index $i\in\{1,\ldots, a-1\}$ and consider all black beads on runner $i$ in nonnegative positions $k = 0,\ldots, d_i -1$. The corresponding SE-steps of $\partial(\kappa)$ enumerated by $j = i,\ i+a, \ldots ,\ i + a(d_i-1)$.

From Proposition~\ref{prop.core}, if $\kappa$ is a $b$- core, then for any positive SE-step $j$ in $\partial(\kappa)$ the step $j-b$ is also an SE-step.

If $i \ge r$, steps $j-b = (i-r) + a(k-s)$ with $k = 0,\ldots,d_i-1$ correspond to the black beads on the runner $(i-r)$ in positions $k = -s,\ -s+1,\ldots,\ d_i-s-1$, so the number of nonnegative black beads on the runner $(i-r)$ is greater or equal to $(d_i -s)$, and thus $d_{i} \leq d_{i-r} + s$.

Similarly, if $i<r$, steps $j-b = (i+a-r) + a(k-s-1)$ with $k = 0,\ldots,d_i-1$ correspond to the black beads on the runner $(i+a-r)$ in positions $k = -s-1,\ -s,\ldots,\ d_i-s-2$, so the number of nonnegative black beads on the runner $(i+a-r)$ is greater or equal to $(d_i - s-1)$, and thus $d_{i} \leq d_{i+a-r}+s+1$.

($\Leftarrow$) Conversely, consider any SE-step $j$ in $\partial(\lambda)$ with the corresponding black bead on runner $i$ and position $k < d_i$.

If $i \ge r$, the step $j-b$ corresponds to a bead on runner $(i-r)$ and position $(k-s)$. Since $k-s < d_i-s \le d_{i-r}$, that bead must be a black one and the step $(j-b)$ is an SE-step. 

If $i<r$, the step $j-b$ corresponds to a bead on runner $(i+a-r)$ and position $(k-s-1)$. Since $k-s-1 < d_i-s-1 \le d_{i+a-r}$, that bead must be a black one and the step $j-b$ is an SE-step.

\end{proof}

Together with Proposition~\ref{prop.distinct}, Proposition~\ref{prop.sim} gives a complete description of the simultaneous $(a,b)$-cores with distinct parts.

\section{Maximum size of $(a,a+1)$-cores with distinct parts.}
\label{sec.s=1}

Combining Proposition~\ref{prop.distinct} and Proposition~\ref{prop.sim} for the case $b=a+1$, we get the following result.

\begin{theorem}
\label{thm.s=1.abacus}
An $a$-core $\kappa$ is an $(a,a+1)$-core with distinct parts if and only if $d=\mathbf{abac}(\kappa)$ has entries $d_i \in \{0,1\}$ and the support set $\mathrm{supp}(d) = \{i: d_i = 1\}$ is an $a$-sparse set.
\end{theorem}
\begin{proof}
From Proposition~\ref{prop.distinct}, the fact that $\kappa$ has distinct parts is equivalent to the sparsity of the set $\mathrm{supp}(d)$.
Now, taking $s=1$ and $r=1$ in Proposition~\ref{prop.sim}, $\kappa$ is an $a+1$-core if and only if $d_i \le d_{i-1} + 1$ for all $i = 1,\ldots a-1$ (condition $d_0 \le d_{a-1} +2$ is always satisfied).

If $i$ is not in $\mathrm{supp}(d)$, then $d_i = 0$ and the equation $d_i \le d_{i-1} + 1$ is true.
If $i$ is in $\mathrm{supp}(d)$, then $d_{i-1} = 0$ because of the sparsity of $\mathrm{supp}(d)$ and so $d_i \le d_{i-1} + 1$ is equivalent to $d_i=1$.
\end{proof}

\begin{theorem}
\label{thm.s=1.fib}
The number of $(a,a+1)$-cores with distinct parts is equal to the Fibonacci number $F_{a+1}$.
\end{theorem}

\begin{proof}
By Theorem~\ref{thm.s=1.abacus}, all $(a,a+1)$-cores are in a bijection with $a$-sparse sets $S = \mathrm{supp}(\mathbf{abac}(\kappa)) \subset \{1, \ldots, a-1\}$. Let the number of $a$-sparse sets be $G_a$. Then, depending on whether an element $a-1$ is in a set, we can divide all $a$-sparse sets into two classes, so that $G_a = G_{a-1} + G_{a-2}$ and $G^{(1)}=1, \ G_2 =2$. Thus $G_a = F_{a+1}$.
\end{proof}

H.Xiong~\cite{Xiong.15} proved Theorem~\ref{thm.s=1.fib} together with conjectures about the largest size of $(a,a+1)$-cores with distinct parts and the number of such cores of maximal size (see Theorem~\ref{thm.s=1}). Here we provide another proof, which it is formulated in a different framework and which also will be useful for our future discussion.

\begin{theorem}
The largest size of an $(a,a+1)$-core with distinct parts is $\big\lfloor\frac{1}{3} \binom{a+1}{2}\big\rfloor$. Moreover, the core of maximal size is unique whenever $(a \mod 3)$ is $0$ or $2$ and there are two cores of maximal size when $(a \mod 3)$ is $1$.
\end{theorem}

\begin{proof}
Given an $(a,a+1)$-core $\kappa$ with distinct parts, take $d = \mathbf{abac}(\kappa)$, $\mathrm{supp}(d) = \{i_1,\ldots, i_n\}$, where $n = n(d)= |\mathrm{supp}(d)| = area(\kappa)$ and indexes $0< i_1 <\ldots <i_n < a$. Denote the gaps between $i_{j}$ and $i_{j+1}$ as $g'_0 = i_1, \ g'_j = i_{j+1} - i_j -1$ for $j =1,\ldots, n-1$ and $g'_n = a-1-i_n$. 

Since $S$ is an $a$-sparse set, $g'_j \geq 1$ for $j=0, 1,\ldots,\ n-1$, and thus we can instead consider nonnegative integer sequence $g_j = g'_j - 1$ for $j = 0,\ldots,\ n-1$ and $g_n = g'_n$. Notice that $\sum_{j=0}^n g'_j = a-n$ and $\sum_{j=0}^n g_j = a-2n$.

\begin{lemma}
\begin{equation}
\label{equation.size}
\mathrm{size}(\kappa) = \frac{1}{6} 3n (2a+1-3n) - \sum_{j=0}^n{j g_j}.
\end{equation}
\end{lemma}
\begin{proof}
Following the construction of an abacus diagram (see Fig.~\ref{figure.abacus}), each row $row_{n-j+1}$ of the partition $\kappa$ is bordered by an SE-step $se_j\in\partial(\kappa)$, which in turn corresponds to a black bead on the runner $i_j$ of the abacus diagram.
 
The length of that row $\kappa_{n-j+1}$ is determined by the number of NE-steps in $\partial(\kappa)$ before the step $se_j$. In the abacus diagram, those NE-steps would correspond to the white beads on runners $i = 0,\ldots,\ i_j-1$ in position 0, and the number of those white beads is equal to $\sum_{k=1}^{j-1} g'_k$.

Summing over all $j$,
\begin{multline*}
	\mathrm{size}(\kappa) = \sum_{j=1}^n \kappa_{n-j+1} =  \sum_{j=1}^n \sum_{k=0}^{j-1} g'_k = \sum_{j=0}^n (n-j) g'_j = n\sum_{j=0}^n g'_j - \sum_{j=0}^n j g'_j =\\ = n(a-n) - \frac{n(n-1)}{2} - \sum_{j=0}^n j g_j = \frac{1}{6} 3n (2a+1-3n) - \sum_{j=0}^n{j g_j}.
\end{multline*}
\end{proof}

Thus, to find a core of largest size, we maximize over $n$ and all nonnegative integer sequences $\{g_j\}_{j=0}^n$ with $\sum g_j = a-2n$.
\begin{equation*}
\max_{\kappa} \mathrm{size}(\kappa) = \max_n \max_{g_j \geq 0} \Big( \frac{1}{6} 3n (2a+1-3n) - \sum_{j=0}^n j g_j\Big) = \max_n \Big(\frac{1}{6} 3n (2a+1-3n) - \min_{g_j \geq 0} \sum_{j=0}^n j g_j\Big).
\end{equation*}

The minimum of $\sum j g_j$ over nonnegative sequences $\{g_j\}_{j=0}^n$ with $\sum g_j = a-2n$ is equal to 0 and is uniquely achieved when $g_0 = a-2n$ and $g_j = 0$ for $j \neq 0$. Thus,

\begin{equation*}
\max_{\kappa} \mathrm{size}(\kappa) = \max_n \Big(\frac{1}{6} 3n (2a+1-3n)\Big).
\end{equation*}

The maximum of the parabola on the right-hand side is achieved at a closest integer point to a number $\frac{2a+1}{6}$.

\begin{enumerate}
\item When $(a\mod 3) = 0$, the maximum is achieved at the unique point $n = \frac{a}{3}$ and the value of the maximum is $\frac{1}{3} \binom{a+1}{2}$.
\item When $(a \mod 3) = 2$, the maximum is achieved at the unique point $n = \frac{a+1}{3}$ and the value of the maximum is $\frac{1}{3} \binom{a+1}{2}$.
\item When $(a \mod 3) = 1$, there are two integer points equally close to a number $\frac{2a+1}{6}$, which are $n_1 = \frac{a-1}{3}$ and $n_2 = \frac{a+2}{3}$. Both integers give the maximum value equal to $\frac{1}{3} \frac{(a-1)(a+2)}{2} = \big\lfloor \frac{1}{3} \binom{a+1}{2} \big\rfloor$.
\end{enumerate}
\end{proof}

In Section~\ref{sec.r=1} we give a generalization of the Proposition~\ref{thm.s=1.abacus} and provide another proof of part (4) of Theorem~\ref{thm.s=1}. Before we do that, however, we need to develop a notion of graded Fibonacci numbers in Section~\ref{sec.r=1}.

\section{Number of $(2k-1,2k+1)$-cores with distinct parts.}
\label{sec.r=2}

The number of $(2k-1,2k+1)$-cores with distinct parts was conjectured by A.Straub~\cite{Straub.16} to be $2^{2k-2}$. That conjecture have been proved by  Yan, Qin, Jin and Zhou in \cite{YQJZ.16} using surprisingly deep arguments. Here, we present another perspective on $(2k-1,2k+1)$-cores using their connection with Dyck paths.

\begin{theorem}
The number of $(2k-1,2k+1)$-cores with distinct parts is equal to $2^{2k-2}$.
\end{theorem}

\begin{proof}
We will use the Dyck path interpretation of cores. For a Dyck path $\pi$, we denoted $\alpha(\pi)$ to be the set of ranks of the area boxes of $\pi$ (see Fig.~\ref{figure.anderson}). We'll also make use of a standard notation $[n] = \{1,\ 2,\ldots,\ n\}$

\begin{lemma}
Under the bijection $\mathbf{path}$, the set of $(2k-1,2k+1)$-cores with distinct parts maps to the set of $(2k-1,2k+1)$-Dyck paths $\pi$ such that $\alpha(\pi) \cap [2k-1]$ is a sparse set. 
\end{lemma}

\begin{proof}[Proof of the lemma]
Given an $(a,b)$- core $\kappa$ with distinct parts and $\pi = \mathbf{path}(\kappa)$, the sparse set $\beta(\kappa)$ is equal to $\alpha(\pi)$. Therefore, we need to prove that the sparsity of $\alpha(\pi)$ is implied by the sparsity of $\alpha(\pi)\cap [2k-1]$. 

Suppose for the sake of contradiction that $\alpha(\pi)\cap [2k-1]$ is sparse and there are two elements $j$ and $j+1$ in $\alpha(\pi)$. Since $\alpha(\pi)$ is a $2k-1$-nested set, elements $(j~\text{mod}~2k-1)$, $(j+1~\text{mod}~2k-1)$ are in $\alpha(\pi)$, they are both in $[2k-1]$, and they differ by one (since there are no multiples of $2k-1$ in $\alpha(\pi)$). Therefore, we get a contradiction.
\end{proof}

Denote $T_k$ to be the upper triangle of $R_{2k-1,2k+1}$, i.e. $T_k$ consists of all boxes with positive rank in $R_{2k-1,2k+1}$ (see Fig.~\ref{figure.anderson}). Separate $T_k$ into three parts by a vertical line $x=k-1$ and a horizontal line $y=k+2$. Below the line $y=k+2$ the boxes of $T_k$ form a staircase-like shape \defn{$A$} that contains, among other boxes, the boxes of odd contents from $[2k-1]$. To the right of the line $x=k-1$ the boxes of $T_k$ form a staircase shape \defn{$B$} that contains the boxes of even contents from $[2k-1]$. Above $y=k+2$ and to the left of $x=k-1$ the boxes of $T_k$ form a square shape \defn{$C$} of size $k-1$.

Now we reflect the shape $A \cup C$ over the main diagonal $y=x$ and denote the resulting shape as \defn{$A^T \cup C^T$}. Put that shape to the right of $C \cup B$ to form a rectangular region \defn{$P_k$} $ = C \cup B \cup A^T \cup C^T$ (see Fig.~\ref{figure.2k-1_2k+1}).

\begin{figure}[t]
\includegraphics[scale=0.18]{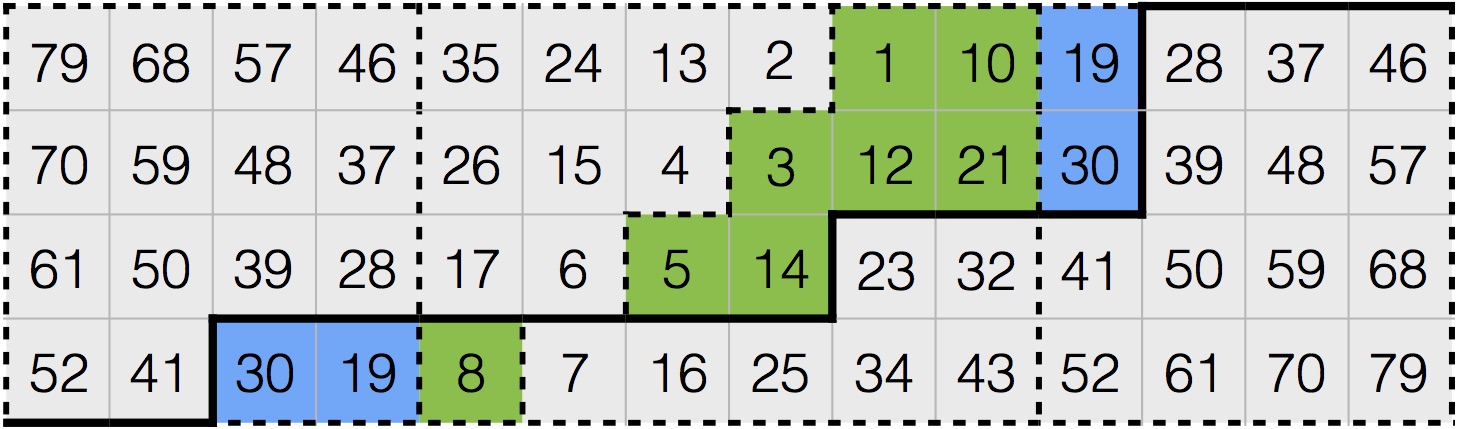} \centering
\caption{The rectangle $P_{5}$ with a $C$-symmetric path $\zeta$. Here, $B(\zeta)=\{8\}$, $A^{T}(\zeta) = \{1,3,5,10,12,14,21\}$ and $C(\zeta) = C^{T}(\zeta) = \{19,30\}$. Also, $i(\zeta) = 1$ and $j(\zeta) = 0$.} \centering
\label{figure.2k-1_2k+1}
\end{figure}

We call a boundary between $B$ and $A^T$ to be the main diagonal of $P_k$. Consider the paths $\zeta$ from the SW corner of $P_k$ to the NE corner of $P_k$ consisting of N and E steps. Denote \defn{$C(\zeta)$} to be the set of contents of boxes below $\zeta$ in $C$, denote \defn{$B(\zeta)$} to be the set of contents of boxes below $\zeta$ in $B$, denote \defn{$A^T(\zeta)$} to be the set of contents of boxes \textit{above} $\zeta$ in $A^T$ and denote \defn{$C^T(\zeta)$} to be the set of contents of boxes \textit{above} $\zeta$ in $C^T$.

We call $\zeta$ to be $C$-symmetric when $C(\zeta) = C^T(\zeta)$ (see Fig.~\ref{figure.2k-1_2k+1} and compare with Fig.~\ref{figure.anderson}).

\begin{lemma}
The set of $C$-symmetric paths $\zeta$ in $P_k$ is in bijection \defn{$\phi$} with the set of  $(2k-1,2k+1)$-Dyck paths with sparse $\alpha(\pi) \cap [2k-1]$. Moreover, $\alpha(\phi(\zeta)) = C(\zeta) \cup B(\zeta) \cup A^T(\zeta)$.
\end{lemma}

\begin{proof}[Proof of the lemma.]
We can define $\phi$ by the property above: $\phi(\zeta) = \pi$ if and only if $\alpha(\pi) = C(\zeta) \cup B(\zeta) \cup A^T(\zeta)$.
First, we need to make sure the map $\phi$ is well-defined, i.e. the set $\gamma(\zeta) := C(\zeta) \cup B(\zeta) \cup A^T(\zeta)$ is a $(2k-1,2k+1)$-nested set (i.e. check conditions (\ref{equation.nested})). 

Let $i \in \gamma(\zeta)$. If $i\in B(\zeta)$ or $i\in A^T(\zeta)$, conditions (\ref{equation.nested}) are satisfied since $B(\zeta)$ and $A^T(\zeta)$ are nested sets by construction. If $i\in C(\zeta) = C^T (\zeta)$, then $i-(2k+1) \in C(\zeta)\cup B(\zeta)$, since $C(\zeta)\cup B(\zeta)$ is $(2k+1)$-nested and $i-(2k-1) \in A^T(\zeta)\cup C^T(\zeta)$ since $A^T(\zeta)\cup C^T(\zeta)$ is $(2k-1)$-nested.

Second, we need to check that $\alpha(\phi(\zeta))\cap [2k-1] = \gamma(\zeta)\cap [2k-1]$ is sparse. Consider $i \in (B(\zeta)\cup A(\zeta) )\cap [2k-1]$, and assume without loss of generality that $i\in B(\zeta)$. Then $i$ is even and it is bordering odd boxes $i-1$ and $i+1$ from E and S directions. Since $\zeta$ goes above the box $i$, it can't go below boxes $i-1$ and $i+1$ in $A^T$, and thus $i-1,i+1 \not\in A^T(\zeta)$.
\end{proof}

Now we want to count the number of $C$-symmetric paths $\zeta$ in $P_k$. If $C(\zeta)$ is non-empty, call $C$-shape of $\zeta$ to be the shape of the diagram under $\zeta$ in $C$, and $C^T$-shape of $\zeta$ is defined correspondingly. Denote \defn{$i(\zeta)$} to be the the width of the $C$-shape of $\zeta$ minus 1 (or the height of $C^T$-shape minus 1). Denote \defn{$j(\zeta)$} to be the height of the $C$-shape of $\zeta$ minus 1 (or the width of $C^T$-shape minus 1). 

The number of $C$-symmetric paths with fixed $i(\zeta) = i$ and fixed $j(\zeta) = j$ is the number of possible paths in $C$ times the number of possible paths in $B \cup A^T$, which is equal to $\binom{i+j}{i} \binom{k+1 + (k-3-i-j)}{k+1}$. If $C(\zeta)$ is empty, the number of paths is equal to $\binom{k+1 +(k-1)}{k+1}$. Thus, the total number of paths is
\begin{multline*}
\binom{2k}{k+1} + \sum_{i,j \geq 0} \binom{i+j}{i} \binom{2k-2-(i+j)}{k+1} = \binom{2k}{k+1} + \sum_{i\geq 0}\sum_{j' \geq 0} \binom{j'}{i} \binom{2k-2-j'}{k+1} = \\ 
= \binom{2k}{k+1} + \sum_{i\geq 0} \binom{2k-1}{k+2+i} = \binom{2k-1}{k} + \binom{2k-1}{k+1} + \sum_{i\geq 0} \binom{2k-1}{k+2+i} =\\ 
= \sum_{k \leq i' \leq 2k-1} \binom{2k-1}{i'} = \frac{1}{2} \sum_{0 \leq i' \leq 2k-1} \binom{2k-1}{i'} = 2^{2k-2}.
\end{multline*}
\end{proof}

\section{Graded Fibonacci numbers and $(a,as+1)$-cores with distinct parts.}
\label{sec.r=1}

Unfortunately, for general $b$ there is no easy way to combine Proposition~\ref{prop.distinct} and Proposition~\ref{prop.sim}. However it can be achieved for specific values of $b$.

\begin{theorem}
\label{thm.description}
Let $\kappa$ be an $a$-core and $b=as+1$ for some integer $s$. Then $\kappa$ is an $(a,b)$-core with distinct parts if and only if the abacus vector $d = \mathbf{abac}(\kappa)$ is sparse and $d_i \leq s$ for $i=1,\ldots, a-1$.
\end{theorem}

\begin{proof}
Similar to the proof of Theorem~\ref{thm.s=1.abacus}, we use Proposition~\ref{prop.distinct} to get the sparsity of $\mathrm{supp}(d)$. Moreover, using Proposition~\ref{prop.sim} with $r=1$ we see that $\kappa$ is an $(a,b)$-core if and only if $d_i \le d_{i-1} + s$ for all $i = 1,\ldots a-1$ (condition $d_0 \le d_{a-1} +s+1$ is automatically satisfied since $d_0 = 0$ for any abacus vector $d$).

If $i$ is not in $\mathrm{supp}(d)$, then $d_i = 0$ and the equation $d_i \le d_{i-1} + s$ is true.
If $i$ is in $\mathrm{supp}(d)$, then $d_{i-1} = 0$ because of the sparsity of $\mathrm{supp}(d)$ and so $d_i \le d_{i-1} + s$ is equivalent to $1 \le d_i \le s$.
\end{proof}

We use similar argument for the case $b = as-1$.

\begin{theorem}
\label{thm.description2}
Let $\kappa$ be an $a$-core and $b=as-1$ for some integer $s$. Then $\kappa$ is an $(a,b)$-core with distinct parts if and only if the abacus vector $d = \mathbf{abac}(\kappa)$ is sparse, $d_i \leq s$ for $i\neq a-1$ and $d_{a-1} \le s-1$.
\end{theorem}

\begin{proof}
Again, we use Proposition~\ref{prop.distinct} to get the sparsity of $\mathrm{supp}(d)$, and use Proposition~\ref{prop.sim} with $r = a-1$ to get inequalities $d_i \le d_{i+1} + s$ for $i = 0,\ldots,\ a-2$ and $d_{a-1} \le d_0 + (s-1)$.

If $i$ is not in $\mathrm{supp}(d)$, then $d_i = 0$ and the equation $d_i \le d_{i+1} + s$ is true.
If $i$ is in $\mathrm{supp}(d)$ and $i \neq a-1$, then $d_{i+1} = 0$ because of the sparsity of $\mathrm{supp}(d)$ and so $d_i \le d_{i-1} + s$ is equivalent to $1 \le d_i \le s$.
If $i = a-1$ and $i$ is in $\mathrm{supp}(d)$, note that $d_0$ is always 0, and thus $d_{a-1} \le d_0 + (s-1)$ equivalent to $1 \le d_{a-1} \le s-1$.
\end{proof}

For the further analysis we'll need a generating function of the area statistic of cores $\kappa$.
We will call that function to be a graded Fibonacci number.

\begin{definition}
For two integers $a$ and $b$, the graded Fibonacci number is
\begin{equation}
\label{eq.fib.cores}
	F_{a,b}(q) = \sum_{\kappa} q^{area(\kappa)},
\end{equation}
where the sum is taken over all $(a,b)$-cores $\kappa$ with distinct parts.
\end{definition}

\begin{remark}
\label{remark.Catalan}
If the sum above was taken over all $(a,b)$-cores, we would have obtained a graded Catalan number (see~\cite{Loehr.03}).
\end{remark}

\begin{remark}
We don't require $a$ and $b$ to be coprime. Despite of the fact that the sum would be infinite, the power series would converge for $|q| <1$. For the further analysis of Catalan numbers with $a$, $b$ not coprime, see~\cite{GMV.17}.
\end{remark}

\begin{remark}
\label{remark.s_to_infty}
When we set $s\to\infty$, the set of $(a,as+1)$-cores coveres the set of all $a$-cores with distinct parts. Thus we will also be interested in the limit of $F^{(s)}_a$ when $s\to\infty$.
\end{remark}

In the light of Theorem~\ref{thm.description} from here and until the end of the paper we will only consider the case $b=as+1$ (although all results that follow are applicable in the case $b = as-1$ with minor modifications). To shorten the notation, we define \defn{$F^{(s)}_a$} $ = F_{a, as+1}$.
It is helpful to rewrite the sum ~\eqref{eq.fib.cores} in terms of vectors $d=\mathbf{abac}(\kappa)$.
Denote the set of all $a$-sparse vectors $d = (d_0,\ d_1,\ldots,\ d_{a-1})$ with $d_0 = 0$ and $d_i \le s$ as \defn{$\mathcal{A}^{(s)}_{a}$}.

\begin{theorem}
\begin{equation}
\label{equation.fib.d}
F^{(s)}_a(q) = \sum_{d\in\mathcal{A}^{(s)}_{a}} q^{\sum d_i}.
\end{equation}
\end{theorem}

\begin{proof}
From Proposition~\ref{thm.description}, the map $\mathbf{abac}\colon \kappa \rightarrow d$ is a bijection from the set of all $(a,b)$-cores $\kappa$ with distinct parts to the set of $a$-sparse vectors $d= (d_0, d_1,\ldots,\ d_{a-1})$ with $d_i \le s$ and $d_0 = 0$, i.e. the set $\mathcal{A}^{(s)}_{a}$.

Also note that bijection $\mathbf{abac}$ sends the $area$ statistic of $\kappa$ to the sum $\sum_{i=0} ^{a-1} d_i$, since the number of rows in $\kappa$ is equal to the number of positive SE-steps of $\partial(\kappa)$, which in turn is equal to the number of nonnegative black beads in the abacus diagram of $\kappa$.

Thus,
\begin{equation*}
F^{(s)}_a(q) = \sum_{\kappa} q^{area(\kappa)} = \sum_{d = \mathbf{abac}(\kappa)} q^{\sum d_i} = \sum_{d\in\mathcal{A}^{(s)}_{a}} q^{\sum d_i}.
\end{equation*}
\end{proof}

Justification of the term "graded Fibonacci numbers" comes from the proposition below. We will use a standard notation \defn{$\left(s\right)_q$} $ = 1+q+\ldots+q^{s-1} = \frac{1-q^s}{1-q}$. 

\begin{theorem}
\label{thm.recurrence}
Graded Fibonacci numbers $F^{(s)}_a (q)$ satisfy recurrence relation
\begin{equation}
\label{eq.fib.recurrence}
F^{(s)}_a (q) = F^{(s)}_{a-1} (q) + q \left(s\right)_q F^{(s)}_{a-2} (q)
\end{equation}
with initial conditions $F^{(s)}_{0} (q) = F^{(s)}_{1} (q) = 1$.
\end{theorem}

\begin{proof}
We divide the sum in (\ref{equation.fib.d}) into two parts: one over vectors $d$ with $a-1 \not\in \mathrm{supp}(d)$, and the other over vectors $d$ with $a-1 \in \mathrm{supp}(d)$.
\begin{multline*}
F^{(s)}_a(q) = \sum_{d \in \mathcal{A}^{(s)}_{a}} q^{\sum d_i} = \sum_{\substack{d\in \mathcal{A}^{(s)}_{a}\\ d_{a-1} = 0}} q^{\sum d_i} + \sum_{\substack{d\in \mathcal{A}^{(s)}_{a}\\ d_{a-1} \neq 0}} q^{\sum d_i} =
\\
= \sum_{d\in \mathcal{A}^{(s)}_{a-1}} q^{\sum d_i} + \sum_{d_{a-1} = 1}^{s} q^{d_{a-1}} \sum_{d \in \mathcal{A}^{(s)}_{a-2}} q^{\sum d_i} = F^{(s)}_{a-1} (q) + q \left(s\right)_q F^{(s)}_{a-2} (q).
\end{multline*}

For initial conditions, notice that $F^{(s)}_{1} (q) = 1$, since there is only one $(1,s+1)$-core, which is empty.
The number of $(2,2s+1)$-cores with distinct parts is equal to $s+1$, with corresponding $d_1 = 0,\ 1,\ldots,\ s$, and thus $F^{(s)}_{2} (q) = 1+q \left(s\right)_q$.

Following the recurrence we proved above, we can set $F^{(s)}_{0} (q) = 1$ for all $s$.
\end{proof}

\begin{remark}
Evaluating~\eqref{eq.fib.recurrence} at $q=1$ would give us a recursive relation for the number of $(a,as+1)$-cores with distinct parts. 
\begin{equation*}
F^{(s)}_a (1) = F^{(s)}_{a-1} (1) + s F^{(s)}_{a-2} (1), \qquad F^{(s)}_{0}(1) = F^{(s)}_{1}(1) = 1.
\end{equation*}
In particular, $F^{(1)}_{a}(1) = F_{a+1}$ is a classical Fibonacci number.
\end{remark}

\begin{remark}
In the limit $s \to\infty$, the relation~\eqref{eq.fib.recurrence} has the form
\begin{equation*}
F^{(\infty)}_{a} (q) = F^{(\infty)}_{a-1} (q) + \frac{q}{1-q} F^{(\infty)}_{a-2} (q).
\end{equation*}
In light of Remark~\ref{remark.s_to_infty}, relation above is the recurrence for the generating function of $area(\kappa)$ over all $a$-cores with distinct parts.
\end{remark}

\begin{theorem}
\label{thm.decompose}
\begin{equation}
F^{(s)}_a (q) = \sum_{n=0}^{\lfloor a/2 \rfloor} \left(q \left(s\right)_{q}\right)^{n} \  \binom{a-n}{n}.
\end{equation}
\end{theorem}

\begin{proof}
For a fixed $a$-sparse support set $\mathbf{S}=\mathrm{supp}(d)$, the sum in (\ref{equation.fib.d}) is equal to 
\begin{equation}
\label{equation.fixed_S}
\sum_{\mathrm{supp}(d) = \mathbf{S}} q^{\sum d_i} = \prod_{i\in \mathbf{S}} \sum_{d_i=1}^s q^{d_i} = \left(q \left(s\right)_q\right)^{\left\vert{\mathbf{S}}\right\vert}.
\end{equation} 
For fixed $n = \left\vert{\mathbf{S}}\right\vert$, the number of possible $a$-sparse support sets $\mathbf{S}$ is the number  $n$-element subsets of $\{1,\ 2,\ldots,\ a-1\}$ such that no two elements are neighbouring each other. The number of such subsets is equal to $\binom{a-n}{n}$.

Summing (\ref{equation.fixed_S}) over all $\mathbf{S}$,

\begin{equation*}
F^{(s)}_a(q) = \sum_{\mathbf{S}} \left(q \left(s\right)_q\right)^{\left\vert{\mathbf{S}}\right\vert} = \sum_{n=0}^{\lfloor a/2 \rfloor} \sum_{\left\vert{\mathbf{S}}\right\vert =n} \left(q \left(s\right)_q\right)^{n}=  \sum_{n=0}^{\lfloor a/2 \rfloor} \left(q \left(s\right)_{q}\right)^{n}  \  \binom{a-n}{n}.
\end{equation*}
\end{proof}

\begin{theorem}
The generating function for $F^{(s)}_a(q)$ with respect to $a$ is
\begin{equation}
G^{(s)} (x;q) := \sum_{a=0}^{\infty} x^a F^{(s)}_a (q) = \frac{1}{1-x-q\left(s\right)_q x^2}.
\end{equation}
\end{theorem}
\begin{proof}
We use the recurrence (\ref{eq.fib.recurrence}).
\begin{multline*}
G^{(s)} (x;q) = \sum_{a=0}^{\infty} x^a F^{(s)}_a (q) = 1 + x + \sum_{a=2}^{\infty} x^a F^{(s)}_a (q) =
\\
= 1+x+ \sum_{a=2}^{\infty} x^a F^{(s)}_{a-1} (q) + \sum_{a=2}^{\infty} x^a q \left(s\right)_q F^{(s)}_{a-2} (q) = 
\\
=1+x + x \sum_{a=1}^{\infty} x^a F^{(s)}_a (q) + q\left(s\right)_q x^2 \sum_{a=0}^{\infty} x^a F^{(s)}_a (q) = 
\\
=1+x +x \left(G^{(s)} (x;q) - 1\right) + q\left(s\right)_q x^2 G^{(s)} (x;q) =
\\
= 1 + \left(x+ q\left(s\right)_q x^2\right) G^{(s)} (x;q).
\end{multline*}
Thus
\begin{equation*}
G^{(s)} (x;q) = \frac{1}{1-x-q\left(s\right)_q x^2}.
\end{equation*}
\end{proof}

\begin{remark}
In the limit $s\to\infty$,
\begin{equation*}
F^{(\infty)}_{a} (q) = \sum_{n=0}^{\lfloor a/2 \rfloor} \left(\frac{q}{1-q}\right)^{n} \  \binom{a-n}{n}, \qquad G^{(\infty)} (x,q) = \frac{1}{1-x-\frac{q}{1-q} x^2}.
\end{equation*}
\end{remark}

Now we consider the case $s=1$ to give a proof of part (4) of Theorem~\ref{thm.s=1}.

\begin{remark}
When $s=1$,
\begin{equation}
\label{equation.s=1.sum}
F^{(1)}_{a} (q) = \sum_{n=0}^{\lfloor a/2 \rfloor} q^{n} \  \binom{a-n}{n}, \qquad G^{(1)} (x,q) = \frac{1}{1-x-qx^2},
\end{equation}
and
\begin{equation}
\label{equation.s=1.recurrence}
F^{(1)}_{a} (q) = F^{(1)}_{a-1} (q) + q F^{(1)}_{a-2} (q), \qquad F^{(1)}_{0}(q) = F^{(1)}_{1}(q) = 1.
\end{equation}
\end{remark}

\begin{theorem}
The total sum of the sizes and the average size of $(a,a+1)$-cores with distinct parts are, respectively, given by
\begin{equation}
\sum_{i+j+k=a+1} F_{i} F_{j} F_{k} \quad \mathrm{and} \quad \sum_{i+j+k=a+1} \frac{F_{i} F_{j} F_{k}}{F_{a+1}}.
\end{equation}
\end{theorem}

\begin{proof}
Denote $\Phi_a$ to be the total sum of sizes of $(a,a+1)$-cores with distinct parts. Since the generating function of Fibonacci numbers $\sum_{i=1}^\infty x^{i} F_{i}$ is equal to $\frac{x}{1-x-x^2}$, then in order to prove the theorem it is enough to show that the generating function $\Gamma(x) := \sum_{a=2}^\infty x^{a+1} \Phi_a$ is equal to 
\begin{equation}
\Gamma(x) \stackrel{?}{=} \sum_{a=2}^{\infty} x^{a+1} \sum_{i+j+k=a+1} F_{i} F_{j} F_{k} = \left( \sum_{i=1}^\infty x^i F_i \right)^3 = \left(\frac{x}{1-x-x^2}\right)^3.
\end{equation}
We use the equation (\ref{equation.size}) to find a formula for $\Phi_a$.

\begin{multline}
\label{equation.size.sum}
\Phi_{a} = \sum_{\kappa} \mathrm{size}(\kappa) = \sum_{n}\sum_{\substack{g \\ \sum g_{i} = a-2n}} \left[ \frac{1}{6} 3n (2a+1-3n) - \sum_{i=0}^n{i g_i} \right] =\\
= \sum_{n} \left[ \binom{a-n}{n} \frac{n (2a+1-3n)}{2} - \sum_{\substack{g \\ \sum g_{i} = a-2n}} \sum_{i=0}^{n} i g_{i} \right]
\end{multline}
To evaluate the double sum, we notice that taking $\lambda$ to be a partition with $\lambda = (1^{g_{1}} 2^{g_{2}}\ldots n^{g_{n}})$,

\begin{multline*}
\sum_{\sum g_{i} = a-2n} \sum_{i=0}^{n} i g_{i} = \sum_{\substack{\lambda_{1} \leq n \\ l(\lambda) \leq a-2n}} |\lambda| =\\
= \left(\text{number of}\ \lambda\ \text{in a rectangle}\ n\times (a-2n)\right) \cdot \left(\text{average size of}\ \lambda\ \text{in } n\times (a-2n)\right) .
\end{multline*}

Note that the number of partitions that fit rectangle $n\times (a-2n)$ is equal to the number of paths from the bottom-right corner of the rectangle to the top-left corner, and thus is equal to $\binom{a-n}{n}$. 

The average size of the partition is equal to half of the area of the rectangle $n\times (a-2n)$ because of the symmetry of partitions. Thus, the average size of $\lambda$ is equal to $\frac{n(a-2n)}{2}$.

Therefore,

\begin{equation*}
\sum_{\sum g_{i} = a-2n} \sum_{i=0}^{n} i g_{i}= \binom{a-n}{n}\  \frac{n(a-2n)}{2}
\end{equation*}

Thus, the sum in (\ref{equation.size.sum}) evaluates to

\begin{equation*}
\Phi_a = \sum_{n} \binom{a-n}{n} \frac{n\left(a-(n-1)\right)}{2} = \frac{a}{2} \sum_{n} \binom{a-n}{n}\ n - \frac{1}{2} \sum_{n} \binom{a-n}{n}\ n(n-1).
\end{equation*}

Comparing it with (\ref{equation.s=1.sum}), the sum simplifies to
\begin{equation}
\label{eq.phi}
\Phi_a =\frac{a}{2} F'_{a} (1) - \frac{1}{2} F''_{a} (1),
\end{equation}
where the derivative $F_{a} (q)$ is a short-hand notation for $F^{(1)}_{a} (q)$, and the derivative $F'$ is taken with respect to $q$ (note that $F'_{0} (1) = F'_{1} (1) = 0$).
Using the expression for a generating function $G^{(1)} (x;q)$ in (\ref{equation.s=1.sum}),
\begin{multline*}
\Gamma (x) = \sum_{a=2}^{\infty} x^{a+1} \Phi_a \stackrel{(\ref{eq.phi})}{=} \frac{a}{2} \sum_{a=1}^\infty x^{a+1} F'_{a} (1) - \frac{1}{2} \sum_{a=0}^\infty x^{a+1} F''_{a} (1) =
\\
= \frac{x^2}{2} \sum_{a=1}^\infty a x^{a-1} F'_{a} (1) - \frac{x}{2} \sum_{a=0}^\infty x^{a} F''_{a} (1) =   \frac{x^2}{2} \frac{\partial^2 G^{(1)}(x;1)}{\partial x \partial q}-\frac{x}{2} \frac{\partial^2 G^{(1)}(x;1)}{\partial q^2} \stackrel{(\ref{equation.s=1.sum})}{=}
\\
\stackrel{(\ref{equation.s=1.sum})}{=} \frac{x^2}{2}\frac{2x(1-x-x^2) - 2x^2(-1-2x)}{(1-x-x^2)^3} - \frac{x}{2}\frac{2x^4}{(1-x-x^2)^3}
=\frac{x^3}{(1-x-x^2)^3}.
\end{multline*}
\end{proof}

\section{Bounce statistic and bigraded Fibonacci numbers.}
\label{section.bigraded}

In light of Remark \ref{remark.Catalan}, we can look at the summand of the bigraded Catalan numbers corresponding to the set of $(a,b)$-cores with distinct parts. There is a definition of bigraded Catalan numbers in terms of $(a,b)$-cores directly (see~\cite{ALW.16}), but the skew length statistic has rather complicated expression in terms of abacus vectors $d$. Nevertheless, we can use another pair of statistics on the set of $(a,b)$-Dyck paths in the case $b=as+1$, namely $area$ and $bounce$.

To define the $bounce$ statistic of a $(a,as+1)$-Dyck path $\pi$, first we present the construction of a bounce path for a Dyck path $\pi$ due to N.Loehr~\cite{Loehr.03}. 

We start at the point $(a,as+1)$ of the rectangle $R_{a,as+1}$ and travel in $W$ (West) direction until we hit an $N$ (North) step of $\pi$. Denote $v_1$ to be the number of $W$ steps we did in the process, and travel $w_1 := v_1$ steps in the $S$ (South) direction. 

After that we travel in $W$ direction until we hit an $N$ step of $\pi$ again. Denote $v_2$ to be the number of $W$ steps made this time, and travel in $S$ direction $w_2 := v_2 +v_1$ steps if $s > 1$ or $w_2 := v_2$ steps if $s=1$. 

In general, on $k$-th iteration, after we travel $v_k \ge 0$ steps in $W$ direction before hitting an $N$-step of $\pi$, we then travel in $S$ direction $w_k:=v_k+\ldots+v_{k-s+1}$ steps if $k\geq s$ or $w_k:=v_k+\ldots+v_1$ steps if $k < s$. The bounce path always stays above the main diagonal and eventually hits the point $(0,1)$, where the algorithm terminates (see~\cite{Loehr.03} for details). 

To calculate \defn{bounce} statistic of $\pi$, each time our bounce path reaches an $N$ step $x$ of $\pi$ after traveling $W$, add up the number of squares to the left of $\pi$ and in the same row as $x$. We will call those rows \defn{bounce rows} (see Fig. \ref{figure.bounce}).

\begin{definition} \cite{Loehr.03}
Bigraded rational Catalan number is defined by the equation 
\begin{equation}
C_{a,as+1} (q,t) = \sum_{\pi} q^{area(\pi)} t^{bounce(\pi)},
\end{equation}
where the sum is taken over all $(a,as+1)$-Dyck paths $\pi$.
\end{definition}

\begin{figure}[t]
\includegraphics[scale=0.35]{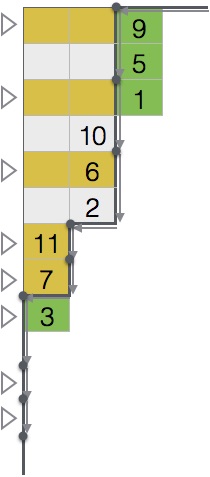} \centering
\caption{$(4,13)$-Dyck path with $\alpha(\pi) = \{1,3,5,9\}$ and $e(\pi) = (3,0,1)$ together with its bounce path. The bounce rows are marked by arrows.} \centering
\label{figure.bounce}
\end{figure}

Following that definition, we restrict the sum to define bigraded Fibonacci numbers.

\begin{definition}
Bigraded rational Fibonacci number is defined by the equation
\begin{equation}
\label{definition.bigraded}
F^{(s)}_a (q,t) = \sum_{\pi} q^{area(\pi)} t^{bounce(\pi)},
\end{equation}
where the sum is taken over all $(a,as+1)$- Dyck paths such that $\kappa = \mathbf{core}(\pi)$ has distinct parts.
\end{definition}

\begin{remark} After the specialization $t=1$ bigraded Fibonacci numbers $F^{(s)}_a (q,1)$ are equal to graded $F^{(s)}_a (q) $ from the previous section. 
\end{remark}

Under the maps $\mathbf{core}$ and $\mathbf{abac}$, there is a correspondence between the Dyck paths $\pi$ in (\ref{definition.bigraded}) and sparse abacus vectors $d \in \mathcal{A}^{(s)}_{a}$. Denote the bounce statistic on $\mathcal{A}^{(s)}_{a}$ to be a bounce statistic of the corresponding Dyck path $\pi$, i.e. $bounce(\mathbf{abac}(\mathbf{core}(\pi))) = bounce(\pi)$.

\begin{theorem}
Given an abacus vector $d\in\mathcal{A}^{(s)}_{a}$,
\begin{equation}
\label{bounce.abac}
bounce(d) = s\ \binom{a}{2} - \sum_{i=1}^{a-1} (a-i)d_i.
\end{equation}
\end{theorem}

\begin{proof}
Let $\pi$ be the corresponding Dyck path, i.e.~$\mathbf{abac}(\mathbf{core}(\pi)) = d$.
It is easier to consider the statistic $bounce'(\pi) = s \binom{a}{2} - bounce(\pi)$. Here $s \binom{a}{2}$ counts the total number of boxes in the upper triangle $T_{a,as+1}$, and thus $bounce'(\pi)$ counts the number of boxes in non-bounce rows to the left of $\pi$ plus the number of area boxes of $\pi$. Thus, we need to show
\begin{equation}
\label{equation.abacus.bounce}
bounce'(d) \stackrel{?}{=}  \sum_{i=1}^{a-1} (a-i)d_i.
\end{equation}

We will prove (\ref{equation.abacus.bounce}) by induction on $a$. Base case $a=1$ is straightforward. Assume now that (\ref{equation.abacus.bounce}) is true for any $a \leq k$ and consider the case $a=k+1$. 

According to Corollary~\ref{path.area.abac}, the area vector $e(\pi) = (d_1,\ldots,d_k)$. 

If $d_1 = e_1(\pi) = 0$, the first $s$ iterations of the bounce path algorithm yields values for $W$ steps $\nu_1=1,\ \nu_2 = \ldots = \nu_s = 0$ and the corresponding steps in $S$ direction are $w_1 = w_2 = \ldots = w_s = 1$, ending at the point $(k,ks+1)$ (and after that point the values of $\nu_1,\ldots,\ \nu_s$ don't contribute to the bounce path). 

Thus the top $s$ rows of the upper triangle $T_{k+1,(k+1)s+1}$ are bounce rows and moreover there are no area boxes of $\pi$ in those rows. Therefore the top $s$ rows don't contribute anything to $bounce'(\pi)$, and we can safely erase them, reducing $a$ by one and reducing all indexes of $d$ by one. Denoting $d' = (d'_0,\ldots,d'_{k-1}) = (d_1,\ldots,d_k) = e(\pi)$,
\begin{equation*}
bounce'(d) = bounce'(d') = \sum_{i=1}^{k-1} (k-i)d'_i = \sum_{i=2}^{k} (k-(i-1))d_i = \sum_{i=1}^{k} ((k+1)-i)d_i.
\end{equation*}

If $d_1= e_1(\pi) >0$ (see Fig.\ref{figure.bounce}), then $d_2 =e_2(\pi) = 0$ because of the sparsity of $d$. The first $s$ iterations of the bounce path algorithm yields values for $W$ steps $\nu_1 =1, \ \nu_2= \ldots = \nu_{s-d_1} = 0,\ \nu_{s-d_1+1} = 1, \ \nu_{s-d_1+2}=\ldots = \nu_{2s-d_1} = 0$ with the corresponding steps in the $S$ direction equal to $w_1 =\ldots = w_{s-d_1} = 1$, $w_{s-d_1 +1} = \ldots = w_s = 2$, $w_{s+1} = \ldots = w_{2s-d_1} = 1$, ending at the point $(k-1,(k-1)s+1)$ (and after that point the values of $\nu_1,\ldots,\ \nu_{2s-d_1}$ don't contribute to the bounce path). 

For the top $2s$ rows of $T_{k+1,(k+1)s+1}$ the non-bounce rows appear exactly when the bounce path travels 2 steps in the $S$ direction, i.e.~when $w_i = 2$. Thus, there are $d_1$ non-bounce rows, each contributing $k-1$ boxes to statistic $bounce'$. Besides that, there are $d_1$ area boxes that also count towards $bounce'$. Thus, the contribution of the top $2s$ rows into $bounce'$ is equal to $d_1k$. Denoting $d' = (d'_0,\ldots,d'_{k-2}) = (d_2,\ldots,d_k)$,
\begin{equation*}
bounce'(d) = bounce'(d') + d_1k = \sum_{i=1}^{k-2} (k-1-i)d_{i+2} + d_1k = \sum_{i=1}^{k} ((k+1)-i)d_i.
\end{equation*}
\end{proof}

\begin{corollary}
\begin{equation}
\label{equation.fibonacci.abacus}
t^{-s\binom{a}{2}} F^{(s)}_a (q,t) =  \sum_{d\in \mathcal{A}^{(s)}_{a}} q^{\sum d_i} t^{-\sum (a-i) d_i} = \sum_{d\in \mathcal{A}^{(s)}_{a}} (qt^{-a})^{\sum d_i} t^{\sum i d_i} = \sum_{d\in \mathcal{A}^{(s)}_{a}} q^{\sum d_i} t^{-\sum i d_i}.
\end{equation}
\end{corollary}

\begin{proof}
First two equalities follow directly from (\ref{definition.bigraded}) and (\ref{bounce.abac}). For the last equality we use the symmetry of $\mathcal{A}^{(s)}_{a}$ under the reflection of indexes $0\mapsto 0, \ i \mapsto a-i$.
\end{proof}

From (\ref{equation.fibonacci.abacus}) it is easier to work with normalized polynomials
\begin{equation}
\label{equation.norm.fibonacci}
\tilde F^{(s)}_a (q,t) := t^{-s\binom{a}{2}} F^{(s)}_a (q,t^{-1}) = \sum_{d\in \mathcal{A}^{(s)}_{a}} q^{\sum d_i} t^{\sum i d_i} = \sum_{d\in \mathcal{A}^{(s)}_{a}} (qt^{a})^{\sum d_i} t^{-\sum i d_i}.
\end{equation}

\begin{remark}
In terms of Dyck paths, $\tilde F^{(s)}_a (q,t)$ is equal to the sum of $q^{area(\pi)} t^{bounce'(\pi)}$ over $(a,as+1)$-Dyck paths $\pi$ with $\mathbf{core}(\pi)$ having distinct parts.
\end{remark}

Using the simple expression of $\tilde F^{(s)}_a (q,t)$ in (\ref{equation.norm.fibonacci}), we prove recursive relations similar to Proposition~\ref{thm.recurrence}.

\begin{theorem}
Normalized bigraded Fibonacci numbers $\tilde F^{(s)}_a (q,t)$ satisfy the following relations:
\begeq
\label{recurrence1}
\tilde F^{(s)}_{a+1} (q,t) = \tilde F^{(s)}_{a} (q,t) + qt^a \left(s\right)_{qt^a} \tilde F^{(s)}_{a-1} (q,t)
\eneq
\begeq
\label{recurrence2}
\tilde F^{(s)}_{a+1} (q,t) = \tilde F^{(s)}_{a} (qt,t) + qt \left(s\right)_{qt} \tilde F^{(s)}_{a-1} (qt^2,t),
\eneq
with initial conditions $\tilde F^{(s)}_0 (q,t) =\tilde F^{(s)}_1(q,t) = 1$.
\end{theorem}
\begin{proof}
For equation \eqref{recurrence1} use the first sum in \eqref{equation.norm.fibonacci} and divide the set $\mathcal{A}^{(s)}_{a+1}$ into two parts corresponding to vectors $d$ with $d_{a} = 0$ and with $d_{a}>0$.
\begin{multline*}
\tilde F^{(s)}_{a+1} (q,t) = \sum_{d\in \mathcal{A}^{(s)}_{a+1}} q^{\sum d_i} t^{\sum i d_i} = 
\sum_{d' \in \mathcal{A}^{(s)}_{a}} q^{\sum d'_i} t^{\sum i d'_i} + \sum_{d_a=1}^s \left(qt^a\right)^{d_a} \sum_{d' \in \mathcal{A}^{(s)}_{a-1}} q^{\sum d'_i}  t^{\sum i d'_i} = \\
= \tilde F^{(s)}_{a} (q,t) + qt^a \left(s\right)_{qt^a} \tilde F^{(s)}_{a-1} (q,t).
\end{multline*}
For equation \eqref{recurrence2} use the second sum in \eqref{equation.norm.fibonacci} and again divide $\mathcal{A}^{(s)}_{a+1}$ into two parts corresponding to vectors $d$ with $d_{a} = 0$ and with $d_{a}>0$.
\begin{multline*}
\tilde F^{(s)}_{a+1} (q,t) = \sum_{d\in \mathcal{A}^{(s)}_{a+1}} (qt^{a+1})^{\sum d_i} t^{-\sum i d_i} =\\
= \sum_{d' \in \mathcal{A}^{(s)}_{a}} \left((qt)t^{a-1}\right)^{\sum d'_i} t^{-\sum i d'_i} + \sum_{d_a=1}^s (qt^{a+1} t^{-a})^{d_a} \sum_{d' \in \mathcal{A}^{(s)}_{a-1}} \left((qt^2)t^{a-2}\right)^{\sum d'_i} t^{-\sum i d'_i} =\\
= \tilde F^{(s)}_{a} (qt,t) + qt \left(s\right)_{qt} \tilde F^{(s)}_{a-1} (qt^2,t)
\end{multline*}
\end{proof}

\begin{remark}
Setting $s=1$ and $a\to\infty$, the recurrence \eqref{recurrence2} gives
\begeq
\tilde F^{(1)}_{\infty} (q,t) = \tilde F^{(1)}_{\infty} (qt,t) + qt \tilde F^{(1)}_{\infty} (qt^2,t),
\eneq
which is an Andrews $q$-difference equation related to Rogers-Ramanujan identities (see~\cite{Andrews.86}).
\end{remark}

\bibliographystyle{acm}
\bibliography{paper}

\begin{thebibliography}{10}

\bibitem{Amdeberhan.15}
{\sc Amdeberhan, T.}
\newblock Theorems, problems and conjectures.

\bibitem{Anderson.02}
{\sc Anderson, J.}
\newblock Partitions which are simultaneously {$t_1$}- and {$t_2$}-core.
\newblock {\em Discrete Math. 248}, 1-3 (2002), 237--243.

\bibitem{Andrews.86}
{\sc Andrews, G.~E.}
\newblock {\em {$q$}-series: their development and application in analysis,
  number theory, combinatorics, physics, and computer algebra}, vol.~66 of {\em
  CBMS Regional Conference Series in Mathematics}.
\newblock American Mathematical Society, Providence, RI, 1986.

\bibitem{AHJ.14}
{\sc Armstrong, D., Hanusa, C. R.~H., and Jones, B.~C.}
\newblock Results and conjectures on simultaneous core partitions.
\newblock {\em European J. Combin. 41\/} (2014), 205--220.

\bibitem{ALW.16}
{\sc Armstrong, D., Loehr, N.~A., and Warrington, G.~S.}
\newblock Rational parking functions and {C}atalan numbers.
\newblock {\em Ann. Comb. 20}, 1 (2016), 21--58.

\bibitem{BNY.17}
{\sc Baek, J., Nam, H., and Yu, M.}
\newblock A bijective proof of amdeberhan's conjecture on the number of
  $(s,s+2)$-core partitions with distinct parts.

\bibitem{GM.16}
{\sc Gorsky, E., and Mazin, M.}
\newblock Rational parking functions and {LLT} polynomials.
\newblock {\em J. Combin. Theory Ser. A 140\/} (2016), 123--140.

\bibitem{GMV.17}
{\sc Gorsky, E., Mazin, M., and Vazirani, M.}
\newblock Rational dyck paths in the non relatively prime case.

\bibitem{GMV.16}
{\sc Gorsky, E., Mazin, M., and Vazirani, M.}
\newblock Affine permutations and rational slope parking functions.
\newblock {\em Trans. Amer. Math. Soc. 368}, 12 (2016), 8403--8445.

\bibitem{Johnson.15}
{\sc Johnson, P.}
\newblock Lattice points and simultaneous core partitions.

\bibitem{Loehr.03}
{\sc Loehr, N.~A.}
\newblock Conjectured statistics for the higher {$q,t$}-{C}atalan sequences.
\newblock {\em Electron. J. Combin. 12\/} (2005), Research Paper 9, 54.

\bibitem{NS.16}
{\sc Nath, R., and Sellers, J.~A.}
\newblock Abaci structures of $(s, ms \pm 1)$-core partitions.

\bibitem{Straub.16}
{\sc Straub, A.}
\newblock Core partitions into distinct parts and an analog of {E}uler's
  theorem.
\newblock {\em European J. Combin. 57\/} (2016), 40--49.

\bibitem{Xiong.15}
{\sc Xiong, H.}
\newblock Core partitions with distinct parts.

\bibitem{YQJZ.16}
{\sc Yan, S. H.~F., Qin, G., Jin, Z., and Zhou, R. D.~P.}
\newblock On {$(2k+1,2k+3)$}-core partitions with distinct parts.
\newblock {\em Discrete Math. 340}, 6 (2017), 1191--1202.

\bibitem{Zaleski.17}
{\sc Zaleski, A.}
\newblock Explicit expressions for the moments of the size of an $(n,
  dn-1)$-core partition with distinct parts.

\bibitem{ZZ.16}
{\sc Zaleski, A., and Zeilberger, D.}
\newblock Explicit (polynomial!) expressions for the expectation, variance and
  higher moments of the size of a $(2n + 1, 2n + 3)$-core partition with
  distinct parts.

\end{thebibliography}

\end{document}